\documentclass[10pt,a4paper,reqno,english]{amsart}
\title[]{A short proof of commutator estimates}
\usepackage{%
    amssymb%
    ,babel%
    ,csquotes%
    ,graphicx%
    ,mathrsfs%
    ,eucal%
    ,xcolor%
    ,enumitem%
    ,geometry%
    ,esint}
\usepackage[%
    pdfauthor={Piero D'Ancona}%
    ,colorlinks=true%
    ,linkcolor=magenta%
    ,citecolor=cyan%
    ,urlcolor=blue%
    ,hyperfootnotes=false%
]{hyperref}

\usepackage{stmaryrd}
  {\baselineskip0pt\normalfont\footnotesize%
  \abovedisplayskip=1pt\abovedisplayshortskip=1pt%
  \belowdisplayskip=1pt\belowdisplayshortskip=1pt%
  \vskip2pt\par\noindent$\llbracket$\ignorespaces}%
  {\ignorespaces$\rrbracket$\vskip2pt}
\numberwithin{equation}{section}
\newtheorem{theorem}{Theorem}[section]

\newtheorem{lemma}[theorem]{Lemma}
\newtheorem{proposition}[theorem]{Proposition}
\theoremstyle{remark}
\newtheorem{remark}[theorem]{Remark}

\theoremstyle{definition}

\date{\today}

\author[P.~D'Ancona]{Piero D'Ancona}
\address{Piero D'Ancona: 
Dipartimento di Matematica\\
Sapienza Universit\`{a} di Roma\\
Piazzale A.~Moro 2\\
00185 Roma\\
Italy}
\email{dancona@mat.uniroma1.it}
\thanks{%
}
\subjclass[2010]{%
}

\begin{document}

\begin{abstract}
  The goal of this note is to give, at least for a restricted range of indices, a short proof of homogeneous commutator estimates for fractional derivatives of a product, using classical tools. Both $L^{p}$ and weighted $L^{p}$ estimates can be proved by the same argument.

  When  the space dimension is 1, we obtain some new estimates in the unexplored range $1/3<r\le1/2$.
\end{abstract}

\maketitle

% e_f_pre  <<<<<<<<< PREAMBOLO

\section{Introduction}\label{sec:intr}

The homogeneous product estimate, 
also called \emph{fractional Leibniz rule}, states that
\begin{equation}\label{eq:leibn}
  \|D ^{s}(uv)\|_{L^{r}(\mathbb{R}^{n})}
  \lesssim
  \|D ^{s}u\|_{L^{p_{1}}}\|v\|_{L^{p_{2}}}
  +
  \|u\|_{L^{q_{1}}}\|D ^{s}v\|_{L^{q_{2}}}
\end{equation}
for $u,v\in \mathscr{S}(\mathbb{R}^{n})$,
where $D ^{s}=(-\Delta)^{s/2}$.
%  is the multiplier operator
% $D ^{s}u=\mathcal{F}^{-1}(|\xi|^{s}\widehat{u}(\xi))$
% with symbol $|\xi|^{s}$. 
The conditions on the indices are
\begin{equation}\label{eq:indL}
  % \textstyle
  \frac1r
  =\frac{1}{p_{1}}+\frac{1}{p_{2}}
  =\frac{1}{q_{1}}+\frac{1}{q_{2}},
  \qquad
  p_{j},q_{j}\in(1,\infty],
  \qquad
  s>\max(0,\frac nr-n)
  \ \text{or}\ 
  s\in 2 \mathbb{Z}^{+}.
\end{equation}
When $r<1$ it is possible to take one of the indices $p_{j}$
(and $q_{i}$)
 equal to 1, provided  the $L^{r}$ 
quasinorm at the left is replaced by a $L^{r,\infty}$ seminorm
(see 
\cite{KatoPonce88-a} and \cite{Ponce91-a}
for the range $1<r<\infty$,
\cite{GrafakosOh14-a} for the extension to values
$r>1/2$, and
\cite{BourgainLi14-a} for the endpoint $r=p_{j}=q_{j}=\infty$). 
The proof relies on 
the Coifman--Meyer theory for bilinear multipliers
\cite{CoifmanMeyer86-a} i.e. on paradifferential methods.
Note that for integer values of $s$ the classical
Gagliardo--Nirenberg estimates are sufficient to 
prove \eqref{eq:leibn}. An analogous estimate holds 
for the non homogeneous case where $D ^{s}$ is replaced by
$J^{s}:=(1-\Delta)^{s/2}$.

Several variants and improvements of \eqref{eq:leibn} are known.
Indeed, the original result of Kato and Ponce 
is the following \emph{commutator estimate} for $s>0$:
\begin{equation}\label{eq:katoponce}
  \|J^{s}(uv)-uJ^{s}v\|_{L^{r}}
  \lesssim
  \|J^{s}u\|_{L^{p_{1}}}\|v\|_{L^{p_{2}}}
  +
  \|\partial u\|_{L^{q_{1}}}\|J^{s-1}v\|_{L^{q_{2}}}.
\end{equation}
An even stronger statement is due to
Kenig, Ponce and Vega \cite{KenigPonceVega93-c}:
\begin{equation}\label{eq:KPV}
  \|D ^{s}(uv)-uD ^{s}v-vD ^{s}u\|_{L^{r}}
  \lesssim
  \|D ^{s_{1}}u\|_{L^{p_{1}}}\|D ^{s_{2}}v\|_{L^{p_{2}}}
\end{equation}
provided
$s=s_{1}+s_{2}$ with $s,s_{j}\in(0,1)$ and
$\frac1r=\frac{1}{p_{1}}+\frac{1}{p_{2}}$ with
$r,p_{1},p_{2}\in(1,\infty)$. Here the additional restrictions
on $s,s_{j}$, are natural, but higher order versions
of \eqref{eq:KPV} have been obtained by D.~Li \cite{Li16-a}
(see also \cite{FujiwaraGeorgievOzawa17-a}).
The paper \cite{Li16-a} gives a comprehensive view of the
state of the art in this genre of inequalities, and in
particular extends \eqref{eq:KPV} to the range
$1/2<r<1$. See also \cite{LenzmannSchikorra16-a} for
an alternative general approach to commutator estimates,
based on the characterizations of functional spaces
given in \cite{BuiCandy17-a}.

We also mention that product and Kato--Ponce commutator
estimates in \emph{weighted} $L^{p}$ spaces, with
Muckenhoupt weights, have been recently proved by
Cruz--Uribe and Naibo in \cite{CruzUribeNaibo16-a},
for the full range of indices \eqref{eq:indL} with the
exception of the endpoint case $r=\infty$.

Our main purpose here is to give a very simple proof of 
the sharper estimate \eqref{eq:KPV}, relying entirely on
classical tools of harmonic analysis. The main drawback is
that in most cases the proof does not cover the full range of 
indices \eqref{eq:indL}. However, in the allowed range of
indices, the method is efficient, and indeed one recovers
both unweighted and weighted estimates with essentially the
same argument. 

The main result of the paper is the following:

\begin{theorem}[]\label{the:commIntro}
  Let $n\ge1$. Assume $s,s_{1},s_{2}$ and $r,p_{1},p_{2}$ satisfy
  \begin{equation*}
    % \textstyle
    s=s_{1}+s_{2}\in(0,2),
    \quad
    s_{j}\in(0,1),
    \qquad
    \frac1r=\frac{1}{p_{1}}+\frac{1}{p_{2}},
    \quad
    \frac{2n}{n+2s_{j}}<p_{j}<\infty.
  \end{equation*}
  Then for all $u,v\in \mathscr{S}(\mathbb{R}^{n})$ we have
  \begin{equation}\label{eq:commestintro}
    \|D ^{s}(uv)-uD ^{s}v-vD ^{s}u\|_{L^{r}}
    \lesssim
    \|D ^{s_{1}}u\|_{L^{p_{1}}}
    \|D ^{s_{2}}v\|_{L^{p_{2}}}
  \end{equation}
  where the $L^{p_{j}}$ norms at the right must be replaced
  by Hardy space norms $H^{p_{j}}$ if $p_{j}\le 1$.

  Moreover, if we define
  \begin{equation*}
    \textstyle
    q_{j}=p_{j}(\frac 12+\frac{s_{j}}{n})
    \quad\text{if $n\ge2$,}
    \qquad 
    q_{j}=\min\{p_{j},p_{j}(\frac 12+s_{j})\}
    \quad\text{if $n=1$,}\quad 
  \end{equation*}
  and we assume in addition $p_{1},p_{2}>1$ when $n=1$,
  then for any weigths $w_{j}\in A_{q_{j}}$ we have
  \begin{equation}\label{eq:commestwintro}
    \|D ^{s}(uv)-uD ^{s}v-vD ^{s}u\|
      _{L^{r}(w_{1}^{r/p_{1}}w_{2}^{r/p_{2}}dx)}
    \lesssim
    \|D ^{s_{1}}u\|_{L^{p_{1}}(w_{1}dx)}
    \|D ^{s_{2}}v\|_{L^{p_{2}}(w_{2}dx)}.
  \end{equation}
\end{theorem}

We briefly discuss the result.

\begin{itemize}
  \item The indices $r,p_{j}$ are in the ranges
  (recall that $s\in(0,2)$ and $s_{j}\in(0,1)$)
  \begin{equation*}
    \textstyle
    \frac{n}{n+s}<r<\infty,
    \qquad
    \frac{2n}{n+2s_{j}}<p_{j}<\infty.
  \end{equation*}
  For $n\ge2$ this is a strict subset of the known set
  given by
  $r\in(\frac12 ,\infty]$, $p_{j}\in(1,\infty]$.
  However, when $n=1$, the ranges are
  \begin{equation*}
    % \textstyle
    \frac{1}{1+s}<r<\infty
    \qquad
    \frac{2}{1+2s_{j}}<p_{j}<\infty
  \end{equation*}
  so that $r$ can be arbitrarily
  close to $\frac13$ and $p_{j}$ to $\frac23$.

  In particular, this shows that the range of indices
  in \cite{GrafakosOh14-a} \emph{is not sharp} and can 
  further be extended.
  This is likely due to the fact that paradifferential
  techniques do not adapt well to the range $p<1$.
  The result suggests that also in dimension $n\ge2$ the usual 
  range of indices may be extended below $1/2$.

  \item 
  The 1--dimensional estimate can be applied also to functions
  of several variables, in the form
  (here 
  $|\partial_{j}|^{s}u=\mathcal{F}^{-1}(|\xi_{j}|^{s}\widehat{u})$)
  \begin{equation*}
    \||\partial_{j}|^{s}(uv)-u|\partial_{j}|^{s}v-v|\partial_{j}|^{s}u\|_{L^{r}}
    \lesssim
    \||\partial_{j}|^{s_{1}}u\|_{L^{p_{1}}}
    \||\partial_{j}|^{s_{2}}v\|_{L^{p_{2}}}
  \end{equation*}
  for $j=1,\dots,n$, and hence also multi--parameter estimates in
  the sense of \cite{MuscaluPipherTao04-a}, 
  \cite{MuscaluPipherTao06-a} can be deduced.

  \item 
  One can deduce from \eqref{eq:commestintro} a fractional
  Leibniz rule via interpolation.
  For instance, in the one dimensional case,
  from \eqref{eq:commestintro} one has
  \begin{equation*}
    \|D^{s}(uv)\|_{L^{r}}
    \lesssim
    \|vD^{s}u\|_{L^{r}}
    +
    \|uD^{s}v\|_{L^{r}}
    +
    \|D^{s_{1}}u\|_{L^{r_{1}}}
    \|D^{s_{2}}v\|_{L^{r_{2}}}.
  \end{equation*}
  To the first two terms one can apply H\"{o}lder's inequality.
  To the third term one applies a
  standard interpolation inequality
  \begin{equation*}
    \|D^{s_{1}}u\|_{L^{r_{1}}}
    \lesssim
    \|D^{s}u\|_{L^{p_{1}}}^{\theta}
    \|u\|_{L^{p_{3}}}^{1-\theta}
  \end{equation*}
  (which follows from the complex interpolation formula
  $\dot H^{s_{1}}_{r_{1}}=[\dot H^{s}_{p_{1}},L^{p_{2}}]_{\theta}$
  with $s_{1}=(1-\theta) s+\theta\cdot0$ and
  $r_{1}^{-1}=(1-\theta) p_{1}^{-1}+\theta p_{2}^{-1}$),
  and a similar one for $v$. Then by Cauchy--Schwartz one obtains
  \eqref{eq:leibn}.

  \item 
  We relax the restriction $s<1$ in
  estimate \eqref{eq:KPV} to $s<2$;
  note that this result is (marginally)
  sharper than the corresponding
  estimates in \cite{Li16-a}.

  \item 
  The weighted estimates \eqref{eq:commestwintro} are new;
  note however that in \cite{CruzUribeNaibo16-a}
  weighted versions of the product estimates \eqref{eq:leibn}
  and of the Kato--Ponce estimates \eqref{eq:katoponce}
  were proved, with conditions on the weights similar to ours.

  \item 
  When $n=1$ and $s>1/2$ , applying a result in
  \cite{Lerner14-a}, we get an explicit and 
  sharp bound of the constant
  in \eqref{eq:commestwintro}, as a function of the Muckenhoupt
  norms of the weights (see Remark \ref{rem:lerner}).
\end{itemize}

The proof is remarkably short and
is based on the explicit representation
\begin{equation}\label{eq:reprc}
  % \textstyle
  D ^{s}(uv)-uD ^{s}v-vD ^{s}u
  =
  c\int \frac{[u(x+y)-u(x)][v(x+y)-v(x)]}{|y|^{n+s}}dy,
  \qquad
  0<s<2
\end{equation}
for a suitable $c=c(n,s)$.  From this we deduce the following
pointwise bound
\begin{equation*}
  \Bigl|D ^{s}(uv)-uD ^{s}v-vD ^{s}u\Bigr|
  \lesssim
  g_{\lambda_{1}}^{*}(D ^{s_{1}}u)(x)
  \cdot
  g_{\lambda_{2}}^{*}(D ^{s_{2}}v)(x),
  \qquad
  s_{1}+s_{2}=s
\end{equation*}
in terms of the Littlewood nontangential square function
$g^{*}_{\lambda}$. In this way classical $L^{p}$ and weighted
$L^{p}$ bounds for $g^{*}_{\lambda}$ can be applied.
The limitations on the set of indices are unavoidable
due to well known counterexamples for the square functions
(see \cite{Fefferman70-a});
it should be possible to obtain
a more complete result by analyzing directly the Dirichlet form
\begin{equation*}
  T_{s}(u,v)=
  % \textstyle
  \int \frac{[u(x+y)-u(x)][v(x+y)-v(x)]}{|y|^{n+s}}dy.
\end{equation*}

\begin{remark}[]\label{rem:besov}
  Note that, thanks to the characterization of homogeneous Besov norms
  \begin{equation*}
    % \textstyle
    \|u\|_{\dot B^{s}_{p,q}}
    =
    \Bigl\|\frac{u(x+y)-u(x)}{|y|^{s+n/q}}
    \Bigr\|_{L^{q}_{y}L^{p}_{x}},
    \qquad
    0<s<1,\qquad p,q\in[1,\infty]
  \end{equation*}
  a Besov version of
  the Kenig--Ponce--Vega estimates \eqref{eq:KPV} is almost
  trivial to prove. Indeed, applying H\"{o}lder's
  inequality to \eqref{eq:reprc} first in $x$ then in $y$, we get
  \begin{equation}\label{eq:besov}
    \textstyle
    \|D ^{s}(uv)-uD ^{s}v-vD ^{s}u\|_{L^{r}}
    \lesssim
    \|u\|_{\dot B^{s_{1}}_{p_{1},q_{1}}}
    \|v\|_{\dot B^{s_{2}}_{p_{2},q_{2}}}
  \end{equation}
  provided $s\in(0,2)$, $s_{j}\in(0,1)$ and
  $r,p_{j},q_{j}\in[1,\infty]$ satisfy
  \begin{equation*}
    % \textstyle
    s=s_{1}+s_{2},
    \qquad
    \frac1r=\frac1{p_{1}}+\frac1{p_{2}},
    \qquad
    1=\frac{1}{q_{1}}+\frac{1}{q_{2}}.
  \end{equation*}
\end{remark}

\begin{remark}[]\label{rem:additional}
  Besides commutator estimates, a similar approach can be
  used to study the \emph{fractional $p$--Laplacian}
  \begin{equation*}
    % \textstyle
    (-\Delta)^{s}_{p}u
    =
    c(n,s,p)
    \int_{\mathbb{R}^{n}}
    \frac{|u(x)-u(y)|^{p-2}[u(x)-u(y)]}{|x-y|^{n+sp}}dy
  \end{equation*}
  and more general Dirichlet forms like
  \begin{equation*}
    % \textstyle
    \mathcal{E}(u,v)=
    \iint
    (u(x)-u(y))(v(x)-v(y))\frac{A(x,y)}{|x-y|^{n+a}}dydx
    \qquad
    a\in(0,2)
  \end{equation*}
  with $\Lambda\ge A(x,y)\ge \Lambda^{-1}>0$,
  see e.g. \cite{IshiiNakamura10-a}, \cite{BassRen13-a}.
\end{remark}

\section{The proofs}\label{sec:proofs}

We begin by recalling the explicit representation
for fractional derivatives as a
\emph{hypersingular integral}, sometimes
named after Aronszajn and Smith:

\begin{lemma}[\cite{AronszajnSmith61-a}]\label{lem:arosm}
  For all $u\in \mathscr{S}(\mathbb{R}^{n})$, $x\in \mathbb{R}^{n}$,
  and $0<s<2$ we have
  \begin{equation}\label{eq:Dsint}
    D ^{s}u(x)=c(n,s)\cdot\lim_{\epsilon \downarrow0}
    \int_{|x|>\epsilon}\frac{u(x+y)-u(x)}{|y|^{n+s}}dy.
  \end{equation}
  % where $c(n,s)=\Gamma(1-s/2)\pi^{-1}2^{-s}\Gamma(s/2)$.
\end{lemma}

\begin{proof}%[of ...]
  Consider the identity
  \begin{equation*}
    % \textstyle
    \Delta_{y}\frac{u(x+y)-u(x)}{|y|^{n+s-2}}
    =
    \frac{\Delta_{y}u(x+y)}{|y|^{n+s-2}}
    -2c
    \nabla_{y}\cdot
    \{y \frac{u(x+y)-u(x)}{|y|^{n+s}}\}
    -sc\frac{u(x+y)-u(x)}{|y|^{n+s}}
  \end{equation*}
  where $c=n+s-2$. If we integrate over 
  $\{y\in \mathbb{R}^{n}:|y|>\epsilon\}$
  and let $\epsilon \downarrow0$
  we obtain
  \begin{equation*}
    % \textstyle
    \int \frac{\Delta_{y}u(x+y)}{|y|^{n+s-2}}dy
    =
    sc \cdot
    \lim_{\epsilon \downarrow0}
    \int_{|x|>\epsilon}\frac{u(x+y)-u(x)}{|y|^{n+s}}dy.
  \end{equation*}
  Since the first integral is precisely
  \begin{equation*}
    % \textstyle
    \int \frac{\Delta_{y}u(x+y)}{|y|^{n+s-2}}dy=
    |\cdot|^{-n-s+2}*\Delta u(x)=
    c'D ^{s-2}\Delta u=c'D ^{s}u
  \end{equation*}
  for a suitable constant $c'=c'(n,s)$, the proof is concluded.
\end{proof}

Thus $D ^{s}u$ can be written as the principal value integral
\eqref{eq:Dsint}; note that in the range $0<s<1$
(and for smooth $u$) the integral is actually absolutely convergent.
In the following we shall write simply
\begin{equation*}
  % \textstyle
  \int
      \frac{u(x+y)-u(x)}{|y|^{n+s}}dy
    \quad\text{instead of}\quad 
    P.V.\int
      \frac{u(x+y)-u(x)}{|y|^{n+s}}dy.
\end{equation*}
Writing $u_{\pm}=u(x\pm y)$, $v_{\pm}=v(x\pm y)$, $u=u(x)$, 
$v=v(x)$, one has the identity
\begin{equation*}
  (u_{+}v_{+}-uv)-u(v_{+}-v)-(u_{+}-u)v=(u_{+}-u)(v_{+}-v),
\end{equation*}
thus \eqref{eq:Dsint} implies the formula
\begin{equation}\label{eq:commint}
  D ^{s}(uv)-uD ^{s}v-vD ^{s}u
  =
  c(n,s)\cdot T_{s}(u,v),
  \qquad
  0<s<2
\end{equation}
where $T_{s}(u,v) $ is the  bilinear form
\begin{equation*}
  T_{s}(u,v)(x)
  =
  \int \frac{[u(x+y)-u(x)][v(x+y)-v(x)]}{|y|^{n+s}}dy,
  \qquad
  0<s<2.
\end{equation*}

\begin{remark}[]\label{rem:absconv}
  It is possible  to work exclusively with absolutely convergent
  integrals, using the equivalent representation
  \begin{equation*}
    % \textstyle
    D ^{s}u(x)=
    c
    \int \frac{u(x+y)+u(x-y)-2u(x)}{|y|^{n+s}}dy
  \end{equation*}
  and the identity
  \begin{equation*}
    (u_{+}v_{+}+u_{-}v_{-}-2uv)
    -u(v_{+}+v_{-}-2v)
    -(u_{+}+u_{-}-2u)v
    =
    (u_{+}-u)(v_{+}-v)+(u-u_{-})(v-v_{-}).
  \end{equation*}
\end{remark}

In order to estimate $T_{s}(u,v)$ we use the
\emph{square fractional integral} (see \cite{Stein61-a})
\begin{equation*}
  % \textstyle
  \mathcal{D}_{\gamma}[u](x)=
  \Bigl(
  \int \frac{|u(x+y)-u(x)|^{2}}{|y|^{n+2\gamma}}dy
  \Bigr) ^{\frac12},
  \qquad
  0<\gamma<1.
\end{equation*}
By Cauchy--Schwartz one has the pointwise bound
\begin{equation}\label{eq:caus}
  |T_{s}(u,v)|
  \le
  \mathcal{D}_{s_{1}}[u]
  \mathcal{D}_{s_{2}}[v],
  \qquad
  s=s_{1}+s_{2},
  \qquad
  s\in(0,2),\ s_{j}\in(0,1)
\end{equation}
and we are reduced to estimate the fractional integral 
$\mathcal{D}_{s}$. To this end, we shall use the
\emph{Littlewood nontangential square function}
$g_{\lambda}^{*}(u)$ defined as follows
($\partial_{x,t}
=(\partial_{x_{1}},\dots,\partial_{x_{n}},\partial_{t})$):
\begin{equation}\label{eq:gla}
  % \textstyle
  g^{*}_{\lambda}(u)(x)
  =
  \Bigl[
  \int_{0}^{\infty}
  \int_{\mathbb{R}^{n}}
  \bigl(\frac{t}{t+|y|}\bigr)^{\lambda n}
  t^{1-n}
  |\partial_{t,x} U(x-y,t)|^{2}
  dydt
  \Bigr]^{\frac12}
\end{equation}
where $U(x,t)$ is the harmonic extension of $u(x)$ in the
upper half space $\mathbb{R}^{n+1}_{x,t}$:
\begin{equation*}
  % \textstyle
  U(x,t)=e^{-tD }u=
  \frac{\Gamma(\frac{n+1}{2})}{\pi^{\frac{n+1}{2}}}
  \int_{\mathbb{R}^{n}}
  \frac{tu(x-y)}{(t^{2}+|y|^{2})^{\frac{n+1}{2}}}dy,
  \qquad
  x\in \mathbb{R}^{n},\ t>0.
\end{equation*}

Now, the crucial step is the following pointwise estimate:

\begin{theorem}[\cite{Stein61-a}]\label{the:dsglest}
  Let $n\ge1$, $0<s<1$ and $\lambda<1+\frac{2s}{n}$. Then we have
  \begin{equation*}
    \mathcal{D}_{s}[u](x)\le c(n,s)g^{*}_{\lambda}(D^{s}u)(x)
  \end{equation*}
  with a constant
  independent of $u\in \mathscr{S}(\mathbb{R}^{n})$ 
  and $x\in \mathbb{R}^{n}$.
\end{theorem}

\begin{proof}%[of ...]
  The result is stated in \cite{Stein61-a} and a 
  hint is given
  in \cite{Stein70-a}. For the sake of competeness, we include
  a proof in the Appendix of the paper.
\end{proof}

Summing up, we have proved the following

\begin{proposition}[Pointiwise commutator estimate]\label{the:ptw}
  Let $n\ge1$ and
  \begin{equation*}
    \textstyle
    s=s_{1}+s_{2}\in(0,2),
    \qquad
    s_{j}\in (0,1),
    \qquad
    \lambda_{j}<1+\frac{2s_{j}}n.
  \end{equation*}
  Then the following  pointwise estimate holds, with a constant
  independent of $u,v\in \mathscr{S}(\mathbb{R}^{n})$ 
  and $x\in \mathbb{R}^{n}$:
  \begin{equation}\label{eq:ptwcru}
    \Bigl|D ^{s}(uv)-uD ^{s}v-vD ^{s}u\Bigr|
    \lesssim
    g_{\lambda_{1}}^{*}(D ^{s_{1}}u)(x)
    \cdot
    g_{\lambda_{2}}^{*}(D ^{s_{2}}v)(x).
  \end{equation}
\end{proposition}

It remains to estimate the square functions at the right
of \eqref{eq:ptwcru}. We recall a few well known 
properties of $g^{*}_{\lambda}$:

\begin{theorem}[]\label{the:gla}
  Let $n\ge1$, $\lambda>1$. 
  For any $u\in \mathscr{S}(\mathbb{R}^{n})$, 
  $g^{*}_{\lambda}(u)$ satisfies the following estimates,
  with constants independent of $u$:
  \begin{enumerate}
  [label=(\roman*)]
    \item $\|g^{*}_{\lambda}(u)\|_{L^{p}}\lesssim\|u\|_{L^{p}}$
    for $\lambda>\max\{1,\frac2p\}$ and $0<p<\infty$,
    where the $L^{p}$ norm at the right must be replaced by
    a Hardy space $H^{p}$ norm if $p\le 1$
    % \item $\|g^{*}_{\lambda}(u)\|_{L^{p}}\lesssim\|u\|_{L^{p}}$
    % for $\lambda>2/p$ and $0<p<2$
    % \item $\|g^{*}_{\lambda}(u)\|_{L^{p,\infty}}\lesssim\|u\|_{L^{p}}$
    % for $\lambda=2/p$ and $0<p<2$
    \item 
    $\|g^{*}_{\lambda}\|_{L^{p}(wdx)}\lesssim\|u\|_{L^{p}(wdx)}$
    for $\lambda>\max\{1,\frac2p\}$, 
    $1<p<\infty$ and $w\in A_{\min\{p,\frac{p \lambda}2\}}$.
  \end{enumerate}
\end{theorem}

\begin{proof}%[of ...]
  Estimate (i) is proved in \cite{Stein61-b}, \cite{Stein70-a},
  \cite{AguileraSegovia77-a}; see also \cite{Torchinsky86-a},
  \cite{MuckenhouptWheeden74-b} for the range $0<p\le1$.
  Estimate (ii) is from \cite{MuckenhouptWheeden74-b}
  (Corollary at p.110).
\end{proof}

In the borderline case $\lambda=2/p$ estimate (i) is valid with
$L^{p,\infty}$ in place of $L^{p}$ at the left
(\cite{Fefferman70-a}, \cite{AguileraSegovia77-a}).
The corresponding weighted weak estimate is contained in
\cite{MuckenhouptWheeden74-b}; moreover, estimate (ii)
holds also for $p\le1$
provided the weighted $L^{p}$ norm at the left  is replaced by a
weighted Hardy space norm  (see \cite{MuckenhouptWheeden74-b}).
The sharp form of the constant in estimate
(ii) is known if $\lambda>2$, $1<p<\infty$:
\begin{equation*}
  \|g^{*}_{\lambda}\|_{L^{p}(wdx)}
  \le C(n,p,\lambda)[w]_{A_{p}}^{\max\{\frac12 ,\frac{1}{p-1}\}}
  \|u\|_{L^{p}(wdx)}
\end{equation*}
as proved in \cite{Lerner14-a}, but it is still unkown for
$\lambda\le2$. Recall that $[w]_{A_{p}}$ for $1<p<\infty$
is the minimal $C$ such that the averages over any ball
$B \subset \mathbb{R}^{n}$ satisfy
\begin{equation*}
  \textstyle
  \fint_{B}w \cdot (\fint_{B}w^{-\frac{1}{p-1}})^{p-1}\le C.
\end{equation*}

Now it is a simple matter to prove the main result:

% \begin{theorem}[Commutator estimate]\label{the:comm}
%   Let $n\ge1$. Assume $s,s_{1},s_{2}$ and $p,p_{1},p_{2}$ satisfy
%   \begin{equation*}
%     \textstyle
%     s=s_{1}+s_{2}\in(0,2),
%     \quad
%     s_{j}\in(0,1),
%     \qquad
%     \frac1r=\frac{1}{p_{1}}+\frac{1}{p_{2}},
%     \quad
%     \frac{2n}{n+2s_{j}}<p_{j}<\infty.
%   \end{equation*}
%   Then for all $u,v\in \mathscr{S}(\mathbb{R}^{n})$ we have
%   \begin{equation}\label{eq:commest}
%     \|D ^{s}(uv)-uD ^{s}v-vD ^{s}u\|_{L^{r}}
%     \lesssim
%     \|D ^{s_{1}}u\|_{L^{p_{1}}}
%     \|D ^{s_{2}}v\|_{L^{p_{2}}}.
%   \end{equation}
%   Moreover, for any weigths $w_{j}\in A_{p_{j}}$ we have
%   \begin{equation}\label{eq:commestw}
%     \|D ^{s}(uv)-uD ^{s}v-vD ^{s}u\|
%       _{L^{r}(w_{1}^{r/p_{1}}w_{2}^{r/p_{2}}dx)}
%     \lesssim
%     \|D ^{s_{1}}u\|_{L^{p_{1}}(w_{1}dx)}
%     \|D ^{s_{2}}v\|_{L^{p_{2}}(w_{2}dx)}.
%   \end{equation}
%   In the one dimensional case $n=1$, if $s_{j}>1/2$ the condition on the
%   corresponding weight can be weakened to
%   $w_{j}\in A_{p_{j}(s_{j}+\frac12)}$.
% \end{theorem}

\begin{proof}[Proof of Theorem \ref{the:commIntro}]
  By H\"{o}lder's inequality and \eqref{eq:ptwcru} we have
  \begin{equation*}
    \|D ^{s}(uv)-uD ^{s}v-vD ^{s}u\|_{L^{r}}
    \lesssim
    \|g^{*}_{\lambda_{1}}(D^{s_{1}} u)\|_{L^{p_{1}}}
    \|g^{*}_{\lambda_{2}}(D^{s_{2}} v)\|_{L^{p_{2}}}
  \end{equation*}
  for any $r,p_{1},p_{2}\in(0,\infty]$ with
  $\frac1r=\frac{1}{p_{1}}+\frac{1}{p_{2}}$,
  any $s_{j}\in(0,1)$
  and any $\lambda_{j}<1+\frac{2s_{j}}{n}$.
  Applying Theorem \ref{the:gla} (i) we get \eqref{eq:commestintro}
  provided we can pick $\lambda_{j}$ such that
  \begin{equation*}
    \textstyle
    \max\{1,\frac{2}{p_{j}}\}<\lambda_{j}<1+\frac{2s_{j}}{n},
  \end{equation*}
  which is possible by the conditions on $p_{j},s_{j}$.
  (Use Hardy norms at the right if $p_{j}\le1$).

  The second estimate is proved in a similar way
  using Theorem \ref{the:gla} (ii).
  We obtain the following condition on the weights:
  \begin{equation*}
    \textstyle
    w_{j}\in A_{q_{j}},\qquad
    1<q_{j}<\min\{p_{j},p_{j}(\frac 12+\frac{s_{j}}{n})\}.
  \end{equation*}
  Thanks to the self improving property of
  Muckenhoupt classes (i.e., if $w\in A_{q}$ with $q>1$ then
  $w\in A_{q_{1}}$ for some $q_{1}<q$), we can relax the 
  condition to
  \begin{equation*}
    \textstyle
    w_{j}\in A_{q_{j}},\qquad
    1<q_{j}=\min\{p_{j},p_{j}(\frac 12+\frac{s_{j}}{n})\}.
  \end{equation*}
  In dimensions $n\ge2$ we have always $s_{j}/n\le1/2$ and
  hence we obtain
  \begin{equation*}
    \textstyle
    w_{j}\in A_{q_{j}},\qquad
    1<q_{j}=p_{j}(\frac 12+\frac{s_{j}}{n}),
  \end{equation*}
  while in dimension $n=1$ we have
  \begin{equation*}
    \textstyle
    w_{j}\in A_{q_{j}},\qquad
    1<q_{j}=\min\{p_{j},p_{j}(\frac 12+s_{j})\}
  \end{equation*}
  and the proof is concluded.
\end{proof}

\begin{remark}[]\label{rem:lerner}
  In the case $n=1$ and $s_{1},s_{2}>1/2$, the values of
  $\lambda_{1},\lambda_{2}$ in the previous proof can be taken
  both $>2$ and then Lerner's result \cite{Lerner14-a} gives
  the following explicit bound on the constant of 
  \eqref{eq:commestwintro}:
  if $p_{1},p_{2}>1$,
  \begin{equation}\label{eq:lernbd}
    C \le c(n,a,s_{j},r,p_{j})
    [w_{1}]_{A_{p_{1}}}^{\max\{\frac12 ,\frac{1}{p_{1}-1}\}}
    [w_{2}]_{A_{p_{2}}}^{\max\{\frac12 ,\frac{1}{p_{2}-1}\}}.
  \end{equation}
\end{remark}

\appendix
\section{Proof of Theorem \ref{the:dsglest}}
\label{sec:appe_proo_of_}

It is sufficient to prove the inequality at $x=0$.
Let $u\in \mathscr{S}(\mathbb{R}^{n})$ and let
$U(x,t)=e^{-tD }u$ be its harmonic extension on 
$\mathbb{R}^{n+1}_{+}$ for $t>0$.
Following the hint in \cite{Stein70-a} V.6.12,
we can estimate the difference $u(y)-u(0)$
with the integral of $\partial_{x,t} U(x,t)$ along any path contained
in $\mathbb{R}^{n+1}_{+}$ joining the points $(0,0)$ and
$(y,0)$. We choose a path made by
a vertical segment joining $(0,0)$ with $(0,|y|)$,
followed by a horizontal segment joining $(0,|y|)$ with $(y,|y|)$,
followed by a vertical segment joining $(y,|y|)$ with $(y,0)$.
We get
\begin{equation*}
  \textstyle
  |u(y)-u(0)|\le
  \int_{0}^{|y|}
  (
  |\partial U(y,\lambda)|+|\partial U(0,\lambda)|+
  |\partial U(\lambda \widehat{y},|y|)|
  )d \lambda
\end{equation*}
where $\widehat{y}=y/|y|$. 
If we denote with
$F(x,t)=e^{-tD }D ^{s}u$ the harmonic extension of
$D ^{s}u$, we have the formula
\begin{equation*}
  \textstyle
  U(z,s)=\int_{0}^{\infty}F(z,t+\mu)\mu^{s-1}d\mu
\end{equation*}
which implies
\begin{equation*}
  \textstyle
  |u(y)-u(0)|\le
  \int_{0}^{|y|}\int_{\lambda}^{\infty}
  (
  |\partial F(y,\mu)|+
  |\partial F(0,\mu)|+
  |\partial F(\lambda \widehat{y},\mu+|y|-\lambda)|
  )(\mu-\lambda)^{s-1}d\mu d \lambda.
\end{equation*}
We split the RHS in the sum of four pieces
$I+II+III+IV$ where
\begin{equation*}
  \textstyle
  I=
  \int_{0}^{|y|}\int_{\lambda}^{|y|}
  |\partial F(y,\mu)|
  (\mu-\lambda)^{s-1}d\mu d \lambda,
\end{equation*}
\begin{equation*}
  \textstyle
  II=
  \int_{0}^{|y|}\int_{\lambda}^{|y|}
  |\partial F(0,\mu)|
  (\mu-\lambda)^{s-1}d\mu d \lambda,
\end{equation*}
\begin{equation*}
  \textstyle
  III=
  \int_{0}^{|y|}\int_{\lambda}^{|y|}
  |\partial F(\lambda \widehat{y},\mu+|y|-\lambda)|
  (\mu-\lambda)^{s-1}d\mu d \lambda,
\end{equation*}
\begin{equation*}
  \textstyle
  IV=
  \int_{0}^{|y|}\int_{|y|}^{\infty}
  (
  |\partial F(y,\mu)|+
  |\partial F(0,\mu)|+
  |\partial F(\lambda \widehat{y},\mu+|y|-\lambda)|
  )(\mu-\lambda)^{s-1}d\mu d \lambda.
\end{equation*}
The term $IV$ can be estimated for any $A$ with
\begin{equation*}
  \textstyle
  IV \lesssim
  \sup\limits_{|z|\le|y|\le \lambda}
  \lambda^{A}|\partial F(z,\lambda)|
  \cdot
  \int_{0}^{|y|}\int_{|y|}^{\infty}
  \mu^{-A}(\mu-\lambda)^{s-1}d\mu d \lambda.
\end{equation*}
The integral is finite if $s<A<1$ and we get
\begin{equation*}
  \textstyle
  IV
  \lesssim
  \sup\limits_{|z|\le|y|\le \lambda}
  \lambda^{A}|y|^{1+s-A}|\partial F(z,\lambda)|.
\end{equation*}
Using the mean value property of the harmonic function $\partial F$,
we get
\begin{equation*}
  \textstyle
  (IV)^{2}
  \lesssim
  \sup\limits_{|z|\le|y|\le \lambda}
  \lambda^{2A}|y|^{2+2s-2A}
  \lambda^{-n-1}
  \int\limits_{D }
  |\partial F(\xi,\tau)|^{2}d \xi d \tau
\end{equation*}
where
\begin{equation*}
  D=
  \{(\xi,\tau):|\xi-z|\le \lambda/2,\ |\tau-\lambda|\le \lambda/2\}.
\end{equation*}
Now we note that
\begin{equation*}
  D \subset
  \{(\xi,\tau)\in \Gamma ,\ \tau\ge|y|/2\}
\end{equation*}
where $\Gamma$ is the cone of aperture 3
\begin{equation*}
  \Gamma=\{(\xi,\tau):|\xi|\le3 \tau\}
\end{equation*}
and moreover, if $(\xi,\tau)\in D$, we have $\tau \simeq \lambda$
and actually $\lambda/2\le \tau\le 3 \lambda /2$.
This gives
\begin{equation*}
  \textstyle
  (IV)^{2}
  \lesssim
  |y|^{2+2s-2A}
  \int_{\Gamma,\  \tau\ge|y|/2}
  |\partial F(\xi,\tau)|^{2}
  \tau^{2A-n-1}
  d \xi d \tau.
\end{equation*}
Now dividing by $|y|^{n+2s}$ and integrating in $y$ we have,
inverting the order of integration,
\begin{equation*}
  \textstyle
  \int \frac{(IV)^{2}}{|y|^{n+2s}}dy
  \lesssim
  \int_{\Gamma}(\int_{|y|\le 2 \tau}|y|^{2-n-2A}dy)
  |\partial F(\xi,\tau)|^{2}\tau^{2A-n-1}d \xi d \tau
\end{equation*}
and finally
\begin{equation}\label{eq:IV}
  \textstyle
  \int \frac{(IV)^{2}}{|y|^{n+2s}}dy
  \lesssim
  \int_{\Gamma}
  |\partial F(\xi,\tau)|^{2}\tau^{1-n}d \xi d \tau.
\end{equation}
(In particular we see that the piece $IV$ is estimated by a
Lusin area integral on a cone of fixed aperture).

The terms $III$ and $II$ satisfy an estimate similar to $IV$.
Indeed, in $III$ we have 
$\mu+|y|-\lambda\ge|y|\ge \lambda=|\lambda \widehat{y}|$
so that
\begin{equation*}
  \textstyle
  III
  \lesssim
  \sup\limits_{|z|\le|y|\le \lambda}
  \lambda^{A}|\partial F(z,\lambda)|
  \cdot
  \int_{0}^{|y|}\int_{\lambda}^{|y|}
  (|y|+\mu-\lambda)^{-A}(\mu-\lambda)^{s-1}d \mu d \lambda;
\end{equation*}
using $(|y|+\mu-\lambda)^{-A}\le|y|^{-A}$ we have
\begin{equation*}
  \textstyle
  \int_{0}^{|y|}\int_{\lambda}^{|y|}
  (|y|+\mu-\lambda)^{-A}(\mu-\lambda)^{s-1}d \mu d \lambda
  \le
  |y|^{-A}\cdot|y|^{1+s}
\end{equation*}
so that
\begin{equation*}
  III
  \lesssim
  \sup\limits_{|z|\le |y|\le \lambda}
  \lambda^{A}|y|^{1+s-A}|\partial F(z,\lambda)|.
\end{equation*}
Proceeds as for $IV$ we obtain that $III$ satisfies \eqref{eq:IV}.

For the term $II$ we have by Fubini and
then by Cauchy--Schwartz, for $\epsilon\in(0,2s)$,
\begin{equation*}
  \textstyle
  II
  \simeq
  \int_{0}^{|y|}|\partial F(0,\mu)|\mu^{s}d\mu
  \lesssim
  (\int_{0}^{|y|}|\partial F(0,\mu)|^{2}\mu^{1+2s-\epsilon}
  d\mu)^{1/2}
  |y|^{\epsilon/2}
\end{equation*}
which gives
\begin{equation*}
  \textstyle
  \int \frac{(II)^{2}}{|y|^{n+2s}}dy
  \lesssim
  \int|y|^{\epsilon-n-2s}\int_{0}^{|y|}
  |\partial F(0,\mu)|^{2}\mu^{1+2s-\epsilon} d\mu
\end{equation*}
\begin{equation*}
  \textstyle
  =
  \int_{0}^{\infty}
  |\partial F(0,\mu)|^{2}\mu^{1+2s-\epsilon}
  \int_{|y|\ge \mu}\frac{dy}{|y|^{n+2s-\epsilon}}
   d\mu
\end{equation*}
so that
\begin{equation*}
  \textstyle
  \int \frac{(II)^{2}}{|y|^{n+2s}}dy \lesssim\int_{0}^{\infty}
  |\partial F(0,\mu)|^{2}\mu d\mu.
\end{equation*}
We now apply the mean value property:
\begin{equation*}
  \textstyle
  \int \frac{(II)^{2}}{|y|^{n+2s}}dy
  \lesssim
  \int_{0}^{\infty}
  \mu^{-n}
  \int_{|\xi|\le\mu/2,|\tau-\mu|\le \mu/2}
  |\partial F(\xi,\tau)|^{2}d \xi d \tau d \mu.
\end{equation*}
Note that the domain of integration of the inner integral is
contained in the cone $\Gamma$ defined above, and that
$\mu \simeq \tau$; more precisely we have
\begin{equation*}
  \textstyle
  \frac \mu 2\le \tau\le \frac{3 \mu}{2}
  \quad\text{i.e.}\quad 
  \frac{2 \tau}{3}\le \mu\le 2 \tau.
\end{equation*}
This gives, after exchanging the order of integration,
\begin{equation*}
  \textstyle
  \int \frac{(II)^{2}}{|y|^{n+2s}}dy
  \lesssim
  \int_{\Gamma}|\partial F(\xi,\tau)|^{2}
  \tau^{-n}\int_{2 \tau/3}^{2 \tau}d \mu d \xi d \tau
\end{equation*}
and finally we get as in \eqref{eq:IV}
\begin{equation}\label{eq:II}
  \textstyle
  \int \frac{(II)^{2}}{|y|^{n+2s}}dy
  \lesssim
  \int_{\Gamma}
  |\partial F(\xi,\tau)|^{2}\tau^{1-n}d \xi d \tau.
\end{equation}

The remaining piece $I$ is the only one requiring cones of
arbitrary aperture, and hence the function $g^{*}_{\lambda}$.
Exchanging the order of integration and using
Cauchy--Schwartz we get
\begin{equation*}
  \textstyle
  I=s^{-1}\int_{0}^{|y|}|\partial F(y,\mu)|\mu^{s}
  d \mu
  \lesssim
  (\int_{0}^{|y|}|\partial F(y,\mu)|^{2}\mu^{1+2s-\epsilon}
  d\mu)^{1/2}
  |y|^{\epsilon/2}
\end{equation*}
for any $\epsilon\in(0,2 s)$. Thus we have
\begin{equation*}
  \textstyle
  \int \frac{(I)^{2}}{|y|^{n+2s}}dy
  \lesssim
  \int\int_{0}^{|y|}
  |\partial F(y,\mu)|^{2}\mu^{1+2s-\epsilon}d\mu
  |y|^{\epsilon-n-2s}dy
\end{equation*}
that is to say
\begin{equation}\label{eq:I}
  \textstyle
  \int \frac{(I)^{2}}{|y|^{n+2s}}dy
  \lesssim
  \int_{|y|\ge\mu}
  |\partial F(y,\mu)|^{2}
  \frac{\mu^{1+2s-\epsilon}}{|y|^{n+2s-\epsilon}}dyd\mu.
\end{equation}

Summing all the pieces, we have proved the estimate
\begin{equation*}
  \textstyle
  \int \frac{|u(x+y)-u_{y}|^{2}}{|y|^{n+2s}}dy
  \lesssim
  \int_{\Gamma}
  |\partial F(\xi,\tau)|^{2}\tau^{1-n}d \xi d \tau
  +
  \int_{|y|\ge\mu}
  |\partial F(y,\mu)|^{2}
  \frac{\mu^{1+2s-\epsilon}}{|y|^{n+2s-\epsilon}}dyd\mu.
\end{equation*}
The first term at the right is obviously bounded by
$g^{*}_{\lambda}(D ^{s}u)(0)$ for all $\lambda$, while it is easy
to check that the second integral is bounded by
$g^{*}_{\lambda}(D ^{s}u)(0)$ provided
$\lambda=1+\frac{2s}{n}-\frac\epsilon n$. Since $\epsilon$ is
arbitraruly small, the proof is concluded.

% b_f_post
% \bibliography{/Users/piero/Documents/Biblioteca/-bib/bibliodatabase.bib}
\bibliographystyle{abbrv}

\end{document}